\documentclass[a4paper]{amsart}
\usepackage{amsmath,amsthm,amssymb}
\usepackage[english]{babel}
\usepackage{enumerate}
\usepackage{microtype}
\usepackage{stmaryrd}
\usepackage{dsfont}

\numberwithin{equation}{section}

\newtheorem{thm}{Theorem}[section]
\newtheorem{cor}[thm]{Corollary}
\newtheorem{lem}[thm]{Lemma}
\newtheorem{prop}[thm]{Proposition}

\theoremstyle{definition}
\newtheorem{defi}[thm]{Definition}

\renewcommand{\phi}{\varphi}

\renewcommand{\leq}{\leqslant}
\renewcommand{\geq}{\geqslant}

\newcommand{\Real}{\mathds{R}}
\newcommand{\Rat}{\mathds{Q}}

\newcommand{\bigset}[1]{\big\{ #1 \big\}}

\begin{document}

\title{Effective Genericity and Differentiability}

\author[R. Kuyper]{Rutger Kuyper}
\address[Rutger Kuyper]{Radboud University Nijmegen\\
Department of Mathematics\\
P.O. Box 9010, 6500 GL Nijmegen, the Netherlands.}
\email{r.kuyper@math.ru.nl}
\thanks{The research of the first author was supported by NWO/DIAMANT grant 613.009.011 and by 
John Templeton Foundation grant 15619: `Mind, Mechanism and Mathematics: Turing Centenary Research Project'.}
\author[S. A. Terwijn]{Sebastiaan A. Terwijn}
\address[Sebastiaan A. Terwijn]{Radboud University Nijmegen\\
Department of Mathematics\\
P.O. Box 9010, 6500 GL Nijmegen, the Netherlands.} 
\email{terwijn@math.ru.nl}

\date{\today}

\begin{abstract}
\noindent 
We prove that a real $x$ is 1-generic if and only if every differentiable computable function 
has continuous derivative at $x$. This provides a counterpart to recent results connecting 
effective notions of randomness with differentiability. 
We also consider multiply differentiable computable functions and polynomial time computable
functions.
\end{abstract}

\maketitle

\section{Introduction}

The notion of 1-genericity is an effective notion of genericity from 
computability theory that has been studied extensively, see e.g.\ 
Jockusch~\cite{Jockusch}, or the textbooks Odifreddi~\cite{OdifreddiII} and 
Downey and Hirschfeldt~\cite{downey-hirschfeldt-2010}.
1-Genericity, or $\Sigma^0_1$-genericity in full, 
can be defined using computably enumerable (c.e.) sets of strings as forcing 
conditions. 
This notion captures a certain type of effective finite extension constructions
that is common in computability theory. 
In this paper we give an characterization of 1-genericity in terms of 
familiar notions from computable analysis. 
This complements recent results by 
Brattka, Miller, and Nies~\cite{BrattkaMillerNies} that characterize 
various notions of algorithmic randomness in terms of computable analysis. 
For example, in \cite{BrattkaMillerNies} it was proven (building on earlier 
work by Demuth~\cite{Demuth}) that an element 
$x\in[0,1]$ is Martin-L\"of random if and only if 
every computable function of bounded variation is differentiable at~$x$.
Note that the notion of Martin-L\"of randomness, which one could also call 
$\Sigma^0_1$-randomness, is the measure-theoretic counterpart of the topological 
notion of 1-genericity. 

The main result of this paper is as follows. 

\begin{thm}\label{thm-main}
A real $x \in [0,1]$ is 1-generic if and only if every 
differentiable computable function $f: [0,1] \to \Real$ has continuous derivative at $x$.
\end{thm}

The two implications of this theorem will be proven in Theorems \ref{thm-main1} and \ref{thm-main2}.
Note that by ``differentiable computable function'' we mean a computable function 
that is classically differentiable, so that in particular the derivative need 
not be continuous. 
Our result can be seen an effectivization of a result by Bruckner and Leonard.

\begin{thm}{\rm (Bruckner and Leonard \cite[p.\ 27]{bruckner-leonard-1966})}
A set $A \subseteq \Real$ is the set of discontinuities of a derivative 
if and only if $A$ is a meager $\mathbf{\Sigma}^0_2$ set.
\end{thm}

One might expect that, in analogy to Theorem~\ref{thm-main}, 
$n$ times differentiable computable functions would characterize $n$-genericity. 
However, in section \ref{sec-mult-dif} we show that $1$-genericity is also equivalent
to the $n$th derivative of any $n$ times differentiable computable function being continuous at $x$.
In section \ref{sec-poly} we consider differentiable polynomial time computable functions and show that
again these characterize 1-genericity.

Our notation is mostly standard. We denote the natural numbers by~$\omega$. 
The Cantor space of all infinite binary sequences is denoted by $2^\omega$, 
and $2^{<\omega}$ is the set of all finite binary strings. 
For a finite string $\sigma$ and a finite or infinite string $x$, 
we denote by $\sigma\sqsubseteq x$ that $\sigma$ is an initial segment of $x$. 
For a string $\sigma\in 2^{<\omega}$, we have 
\[[\sigma] = \bigset{x \in 2^\omega : \sigma\sqsubseteq x}.\]
The product topology on $2^\omega$, sometimes called the tree topology, 
or the finite information topology, has all sets of the form $[\sigma]$ as 
basic open sets. 
For a set $A \subseteq 2^{<\omega}$, we let 
\[[A] = \bigcup_{\sigma\in A} [\sigma].\]
Thus every set $A$ of finite strings defines an open subset of $2^\omega$. 
A subset of $2^\omega$ is a \emph{$\Sigma^0_1$ class}, or \emph{effectively open}, 
if it is of the form $[A]$, with $A\subseteq 2^{<\omega}$ computably enumerable (c.e.). 
A set is a \emph{$\Pi^0_1$ class}, or \emph{effectively closed}, if it is the 
complement of a $\Sigma^0_1$ class. 
Thus $\Sigma^0_1$ and $\Pi^0_1$ classes form the first levels of the 
effective Borel hierarchy. 
As usual, the levels of the classical Borel hierarchy are denoted by 
boldface symbols $\mathbf{\Sigma}^0_n$ and $\mathbf{\Pi}^0_n$.
These notions are defined in the same way for $[0,1]$, using 
rational intervals as basic opens.
We denote the interior of a set $V\subseteq 2^\omega$ by $\mathrm{Int}(V)$. 

For unexplained notions from computability theory, we refer to 
Odifreddi~\cite{OdifreddiI} or Downey and Hirschfeldt~\cite{downey-hirschfeldt-2010}.
For background in descriptive set theory we refer to Kechris~\cite{kechris-1995} or 
Moschovakis~\cite{moschovakis-2009}. 
Further background on (classical) Baire category theory can also be found in 
Oxtoby~\cite{oxtoby-1980}.

\section{1-Genericity}

First, let us recall what it means for an element $x \in 2^\omega$ to be 1-generic.
We will then discuss 1-genericity for elements of $[0,1]$. 
A discussion of the properties of arithmetically generic and 1-generic sets
can be found in Jockusch~\cite{Jockusch}. 
The ``forcing-free'' formulation of genericity we use here is due to 
Posner, see \cite[p115]{Jockusch}.

Given a sequence $x\in 2^\omega$ and a set $A \subseteq 2^{<\omega}$, 
we say that $x$ \emph{meets\/} $A$ if there exists $\sigma\sqsubseteq x$ 
such that $\sigma\in A$;
equivalently, if $x\in [A]$. 
The set $A$ is \emph{dense along\/} $x$ if for every $\sigma\sqsubseteq x$
there is an extension $\tau\sqsupseteq \sigma$ such that $[\tau]\subseteq [A]$;
equivalently, if $x$ is in the closure of the open set $[A]$.
% This equivalence does not hold if you define tau\in A: 
% It may then happen that x in [A] but A is not dense along x. 

\begin{defi} \label{def:1-generic}
An element $x \in 2^\omega$ is \emph{$1$-generic} if $x$ meets every 
c.e.\ set $A \subseteq 2^{<\omega}$ that is dense along $x$. 
\end{defi}

We now reformulate the definition of 1-genericity into a form 
that will be convenient in what follows.
This formulation is also better suited for the discussion of 
generic real numbers (as opposed to infinite strings). 

\begin{lem}\label{lem-gen-int}
Let $A \subseteq 2^{<\omega}$ and let $V = 2^\omega \setminus [A]$. 
Then $A$ is dense along $x$ if and only if $x$ is not in the interior of $V$. 
Therefore, $A$ is dense along $x$ and $x$ does not meet $A$ if and only if 
$x \in V \setminus \mathrm{Int}(V)$.
\end{lem}
\begin{proof}
Note that $A$ is dense along $x$ if and only if every open set containing $x$ has 
nonempty intersection with $[A]$. Thus, $A$ is dense along $x$ if and only if every 
open set disjoint from $[A]$ does not contain $x$. However, the open sets disjoint 
from $[A]$ are exactly the open sets contained in $V$, of which 
$\mathrm{Int}(V)$ is the largest. Thus $A$ is dense along $x$ if and only if 
$\mathrm{Int}(V)$ does not contain $x$. 
\end{proof}

\begin{cor}\label{cor-gen-int}
For any $x \in 2^\omega$ we have that $x$ is 1-generic if and only if 
for every $\Pi^0_1$ class $V \subseteq 2^\omega$ we have $x \not\in V \setminus \mathrm{Int}(V)$.
\end{cor}

%\noindent
One of the reasons this is interesting to mention explicitly is because a
typical example of a nowhere dense set is a closed set with its interior removed,
and the $\Pi^0_1$ sets are the simplest type of closed sets. 
Thus, the Corollary~\ref{cor-gen-int} says that $x$ is 1-generic if it is not in any 
of the simple, typical nowhere dense sets. This way of looking at 1-generic sets complements
the usual motivation of 1-genericity by forcing, and it also allows one to easily compare 1-genericity 
with weak 1-genericity (since $x$ is weakly 1-generic if it is not in any $\Pi^0_1$-class 
with empty interior, see~\cite{downey-hirschfeldt-2010}).

With this equivalence in mind, we can now also define what it means for an 
element of $[0,1]$ to be $1$-generic.

\begin{defi}\label{defi-gen-unit}
Let $x \in [0,1]$. We say that $x$ is \emph{$1$-generic} if for every $\Pi^0_1$ class 
$V \subseteq [0,1]$ we have $x \not\in V \setminus \mathrm{Int}(V)$.
\end{defi}

There is a natural `almost-homeomorphism' between $2^\omega$ and $[0,1]$: 
given an infinite sequence $x \in 2^\omega$ we have $0.x \in [0,1]$
(interpreting the sequence as a decimal expansion in binary), 
and conversely given $y \in [0,1]$ we can take the binary expansion of $y$ 
containing infinitely many 0s, which gives us an element of $2^\omega$. 
Note that the problem of nonunique expansions only occurs for rationals, 
which are not 1-generic anyway. 
It is thus natural to ask if the notions of $1$-genericity in these two spaces 
correspond via this mapping. 
The next proposition says this is indeed the case.

\begin{prop}\label{prop-gen-int}
For any irrational $x \in [0,1]$ we have that $x$ is 1-generic if and only if  
its (unique) binary expansion is $1$-generic in $2^\omega$.
\end{prop}
\begin{proof}
Let $2^\omega_-$ be the set of infinite binary sequences which contain infinitely many 
0s and infinitely many 1s. 
Then the `almost-homeomorphism' given above in fact restricts to a homeomorphism to 
$2^\omega_-$ and $[0,1]_-$, where $[0,1]_-$ is $[0,1]$ without the dyadic rationals. 
Therefore, $1$-genericity on $2^\omega_-$ and $[0,1]_-$ 
(which are defined as in Definition \ref{defi-gen-unit}) coincide.

Note that $2^\omega_-$ is dense in $2^\omega$ and that $[0,1]_-$ is dense in $[0,1]$. 
Thus, it is enough if we can show that if $Y \subseteq X$ is such that $Y$ is dense in $X$, 
then 1-genericity on $X$ and $Y$ coincide for elements $y \in Y$. 
Given a $\Pi^0_1$ class $V \subseteq X$, let $W=V \cap Y$. Then $W$ is a $\Pi^0_1$ class in $Y$. 
Conversely, every $\Pi^0_1$ class $W \subseteq Y$ is of the form $W = V \cap Y$ 
by definition of the induced topology.

We claim that $\mathrm{Int}_X(V) \cap Y = \mathrm{Int}_Y(W)$.
Clearly, $\mathrm{Int}_X(V) \cap Y \subseteq \mathrm{Int}_Y(W)$. Conversely, if we let $\mathrm{Int}_Y(W) = U \cap Y$ for some open $U \subseteq X$, then $U \subseteq \mathrm{Int}_X(V \cup (X \setminus Y))$. Towards a contradiction, assume that $U \cap (X \setminus V) \not= \emptyset$, then this is a nonempty open set. However, we also have $U \cap (X \setminus V) \subseteq X \setminus Y$, which contradicts the fact that $Y$ is dense in $X$. Thus, we see that $U \subseteq V$, and therefore $U \subseteq \mathrm{Int}_X(V)$.
So, $\mathrm{Int}_X(V) \cap Y = \mathrm{Int}_Y(W)$. 

So, we see that $y \not\in V \setminus \mathrm{Int}_X(V)$ if and only if $y \not\in W \setminus \mathrm{Int}_Y(W)$. This completes the proof.
\end{proof}

\section{Effective Baire class 1 functions}

In this section we will discuss what it means for a function to be of effective Baire class 1, 
and discuss some of the basic properties of these functions. First, let us recall what it 
means for a function on the reals to be computable. 
Our definitions follow Moschovakis \cite{moschovakis-2009}.

\begin{defi} \label{def:computable}
Let $f: [0,1] \to \Real$. 
We say that $f$ is \emph{computable} if for every basic open set $U$ we have that $f^{-1}(U)$ is $\Sigma^0_1$ uniformly in $U$, 
i.e.\ if there exists a computable function $\alpha: \Rat \times \Rat \to \omega$ such that 
for all $q,r \in \Rat$ we have that $f^{-1}((q,r))$ is equal to the $\Sigma^0_1$ class given by the index $\alpha(q,r)$.
\end{defi}

%\noindent
Definition~\ref{def:computable} is equivalent to the formulation
with computable functionals, see e.g.\ the discussion in 
Pour-El and Richards~\cite{Pour-ElRichards}. 

Functions of effective Baire class~1 are obtained by 
weakening the above definition as follows. 

\begin{defi}
A function $f: [0,1] \to \Real$ is of \emph{effective Baire class~1} if for every 
basic open set $U$ we have that $f^{-1}(U)$ is $\Sigma^0_2$ uniformly in~$U$.
\end{defi}

Replacing $\Sigma^0_2$ by $\mathbf{\Sigma}^0_2$ in the above definition, 
we obtain what is known as a function of \emph{(non-effective) Baire class 1}. 
Before we give an important example of an effective Baire class 1 function, 
let us first consider the following proposition, which gives an equivalent condition 
for a function to be of effective Baire class~1. This proposition mirrors the classical 
proposition which says that a function is of Baire class 1 if and only if it is a 
pointwise limit of continuous functions, see e.g.\ Kechris \cite[p.\ 192]{kechris-1995}.
(This does not hold for all Polish spaces; it holds for $f:X\rightarrow Y$ if 
either $X$ is zero-dimensional or $Y=\Real$.)

\begin{prop} \label{prop:effectiveBaire1}
Let $f: [0,1] \to \Real$. The following are equivalent:
\begin{enumerate}[\rm (i)]
\item $f$ is of effective Baire class 1,
\item $f$ is the pointwise limit of a uniform sequence of computable functions, 
i.e.\ there exists a sequence $f_0,f_1, \dots$ of functions 
from $[0,1]$ to $\Real$ converging pointwise to $f$ and a computable function 
$\alpha: \omega \times \Rat \times \Rat \to \omega$ such that for all 
$q,r \in \Rat$ and all $n \in \omega$ we have that $f_n^{-1}((q,r))$ is equal to the $\Sigma^0_1$ class 
given by the index $\alpha(n,q,r)$.
\end{enumerate}
\end{prop}
\begin{proof}
(ii) $\to$ (i): 
Let $f_0,f_1,\dots$ be a sequence of uniformly computable functions converging to $f$ and 
let $U$ be any basic open set. Then 
$U = \bigcup_{i \in \omega,V_i \subseteq U} V_i$, 
where $V_0,V_1,\dots$ is a computable enumeration of the closed intervals with rational endpoints.
We claim:
\[
f^{-1}(U) = 
\bigcup_{V_i \subseteq U} \bigcup_{n \in \omega} 
\bigcap_{m \geq n} f_m^{-1}(V_i),
\]
which is clearly $\Sigma^0_2$ uniformly in $U$.

To prove the claim, let $x \in f^{-1}(U)$. Then $f(x) \in U$, so there exists 
$V_i \subseteq U$ such that $f(x) \in \mathrm{Int}(V_i)$, say 
$(f(x)-\varepsilon,f(x)+\varepsilon) \subseteq V_i$. 
Let $n \in \omega$ be such that for every $m \geq n$ we have that $|f_m(x) - f(x)| < \varepsilon$. 
Then for every $m \geq n$ we have that $x \in f_m^{-1}(V_i)$, which proves the first inclusion.

Conversely, let $n \in \omega$, $V_i \subseteq U$ and $x \in \bigcap_{m \geq n} f_m^{-1}(V_i)$. Then for every $m \geq n$ we have $f_m(x) \in V_i$, and since $V_i$ is closed we then also have 
$f(x) =  \lim_{m \to \infty} f_m(x) \in V_i \subseteq U$, which completes the proof of the claim.

(i) $\to$ (ii): 
This follows by effectivizing Kechris \cite[Theorem 24.10]{kechris-1995}; 
this result is also mentioned (without proof) in Moschovakis \cite[Exercise 3.E.14]{moschovakis-2009}. 
Since this implication is not used anywhere in this paper, we will not go into further detail.
\end{proof}

Using this proposition, we can now give an important example of effective Baire class 1 functions: derivatives of computable functions. This also explains our interest in them.

\begin{cor}\label{dif-baire}
Let $f: [0,1] \to \Real$ be a differentiable computable function. Then $f'$ is of effective Baire class~$1$.
\end{cor}
\begin{proof}
Let $f_n(x) = 2^n(f(x + 2^{-n}) - f(x))$.
To account for the problem that for $x + 2^{-n}>1$ the value $f(x + 2^{-n})$ 
is not defined, we let $f(y) = -f(2-y) + 2f(1)$ for $y > 1$ (i.e.\ we flip and mirror $f$ on $[1,2]$).
Then the sequence $f_0,f_1,\dots$ is uniformly computable and converges pointwise to $f'$, 
so $f'$ is of effective Baire class~1 by Proposition~\ref{prop:effectiveBaire1}.
\end{proof}

\section{Continuity of Baire class 1 functions}

At the basis of this section lies the following important classical result.

\begin{thm}{\rm (Baire)}\label{thm-baire}
Let $f: [0,1] \to \Real$ be of (non-effective) Baire class~1. Then the points of discontinuity 
of $f$ form a meager $\mathbf{\Sigma}^0_2$ set.
\end{thm}
\begin{proof}
See Kechris \cite[Theorem 24.14]{kechris-1995} or Oxtoby \cite[Theorem 7.3]{oxtoby-1980}.
\end{proof}

We will now effectivize this result.

\begin{thm}\label{baire-eff}
Let $f: [0,1] \to \Real$ be of effective Baire class~$1$. 
Then $f$ is continuous at every $1$-generic point.
\end{thm}
\begin{proof}
We effectivize the proof from Kechris \cite[Theorem 24.14]{kechris-1995}.
Let $U_0,U_1,\dots$ be an effective enumeration of the basic open sets.
Now $f$ is continuous at $x$ if and only if the inverse image of every 
neighborhood of $f(x)$ is a neighborhood of $x$. Thus, $f$ is discontinuous at $x$ 
if and only if there exists an open set $U$ containing $f(x)$ such that every open 
set contained in $f^{-1}(U)$ does not contain $x$. 
Hence
\[\{x \in [0,1] \mid f \textrm{ is discontinous at $x$}\} = 
\bigcup_{n \in \omega} f^{-1}(U_n) \setminus \mathrm{Int}(f^{-1}(U_n)).\]

Now, let $x$ be such that $f$ is discontinuous at $x$ and let $n$ be such that 
$x \in f^{-1}(U_n) \setminus \mathrm{Int}(f^{-1}(U_n))$. 
Because $f$ is of effective Baire class~$1$, we know that $f^{-1}(U_n)$ is $\Sigma^0_2$. 
So, let $f^{-1}(U_n) = \bigcup_{i \in \omega} V_i$, where each $V_i$ is $\Pi^0_1$. 
Then it is directly verified that
\[f^{-1}(U_n) \setminus \mathrm{Int}(f^{-1}(U_n)) \subseteq \bigcup_{i \in \omega} (V_i \setminus \mathrm{Int}(V_i)).\]
Let $i$ be such that $x \in V_i \setminus \mathrm{Int}(V_i)$. Then $x$ is not 1-generic by 
Definition~\ref{defi-gen-unit}.
\end{proof}

Combining this result with the fact that derivatives of computable functions are of effective Baire class 1, we get the first implication of Theorem \ref{thm-main} as a consequence.

\begin{thm}\label{thm-main1}
If $f: [0,1] \to \Real$ is a computable function, then $f'$ is continuous at every 1-generic real.
\end{thm}
\begin{proof}
From Corollary \ref{dif-baire} and Theorem \ref{baire-eff}.
\end{proof}

\section{Functions discontinuous at non-1-generics} \label{sec:func-cons}

In this section we will prove the second implication of Theorem \ref{thm-main}. 
To this end, we will build, for each
$\Pi^0_1$ class $V$, a Volterra-style differentiable 
computable function whose derivative will fail to be continuous at the points 
whose non-1-genericity is witnessed by $V$. We have to be careful in order
to make this function computable.

\begin{thm}\label{func-cons}
Let $V$ be a $\Pi^0_1$ class. 
Then there exists a differentiable computable function 
$f: [0,1] \to \Real$ such that $f'$ is discontinuous at every  
$x \in V \setminus \mathrm{Int}(V)$.
\end{thm}
\begin{proof} In the construction of $f$ below, we first 
define auxiliary functions $g$ and $h$. 

\textit{Construction.}
We define an auxiliary function $g$, with the property that 
$g$ is differentiable and computable, and $g'$ is continuous on $(0,1)$ 
and discontinuous at $0$ and~$1$. 

Define the function $h$ on $[0,1]$ by $h(0)=0$ and 
\[
h(x) = x^2 \sin(\frac{1}{x^2})
\]
for $x>0$. 
Then $h$ is computable and differentiable, with derivative $h'(0)=0$ and 
\[
h'(x) = 2x\sin(\frac{1}{x^2}) - 2\frac{1}{x}\cos(\frac{1}{x^2})
\]
when $x>0$. Note that $h'$ is discontinuous at $x=0$. 
Fix a computable 
$x_0 \in (0,\frac{1}{2}]$ such that 
$h'(x_0)=0$.  
Such an $x_0$ exists, because $h'$ has isolated roots, 
and isolated roots of computable functions are computable. 
Now define $g$ on $[0,1]$ by
\begin{equation*}
g(x) = \begin{cases} 
0      & \text{if } x=0\\
h(x)   & \text{if } x \in (0,x_0]\\
h(x_0) & \text{if } x \in [x_0,1-x_0]\\
h(1-x) & \text{if } x \in [1-x_0,1)\\
0      & \text{if } x=1.
\end{cases}
\end{equation*}
Then $g$ is a differentiable computable function, with derivative
\begin{equation*}
g'(x) = \begin{cases} 
0        & \text{if } x=0\\
h'(x)    & \text{if } x \in (0,x_0]\\
0        & \text{if } x \in [x_0,1-x_0]\\
-h'(1-x) & \text{if } x \in [1-x_0,1)\\
0        & \text{if } x=1.
\end{cases}
\end{equation*}
In particular, we see that $g'$ is continuous exactly on $(0,1)$. 
We will use $g$ to construct~$f$.

For the given $\Pi^0_1$ class $V$, let $U = [0,1] \setminus V$, and
fix computable enumerations $q_0,q_1,\dots$ and $r_0,r_1,\dots$ of rational numbers in $[0,1]$ 
such that $U = \bigcup_{n \in \omega} [q_n,r_n]$ and such that the $(q_n,r_n)$ are pairwise disjoint.
We will construct $f$ as a sum of a sequence $f_0,f_1,\dots$ of uniformly computable functions. 
We define $f_n$ by:
\begin{equation} \label{eq:f}
f_n(x) = \begin{cases} 
0 & \text{if } x \in [0,q_n]\\
\displaystyle
\frac{r_n-q_n}{2^n} \, g\left(\frac{x-q_n}{r_n-q_n}\right) & \text{if } x \in [q_n,r_n]\\
0 & \text{if } x \in [r_n,1].
\end{cases}
\end{equation}
Finally, we let $f = \sum_{n=0}^\infty f_n$.

\bigskip
\textit{Verification.}
We first show that $f$ is computable. 
To this end, first observe that each $f_n$ is 
supported on $(q_n,r_n)$, and therefore the supports of the different $f_n$ are disjoint. 
Furthermore, each $f_n$ is bounded by $2^{-n}$.

Let $(a,b)$ be a basic open subset of $\Real$. 
We distinguish two cases. First, assume $0 \not\in (a,b)$. We assume $a > 0$, 
the case $b < 0$ is proven in a similar way. Let $n \in \omega$ be such that $2^{-n} < a$. 
Then, since the supports of the $f_m$ are disjoint, and each $f_m$ is bounded by $2^{-m}$, 
we have 
\[f^{-1}((a,b)) = (f_0 + \dots + f_n)^{-1}((a,b)),\]
which is $\Sigma^0_1$ because a finite sum of computable functions is computable.

In the second case, we have $0 \in (a,b)$. Let $n$ be such that $|a|,|b| \geq 2^{-n}$. 
Then, again because the supports of the $f_m$ are disjoint, we see that if $x$ is not in the
support of any $f_m$ for $m \leq n$ then certainly $f(x) \in (a,b)$. Therefore we have
\[
f^{-1}((a,b)) = (f_0 + \dots + f_{n})^{-1}((a,b)) \cup 
\bigcap_{m\leq n} ([0,1] \setminus [q_m,r_m]),
\]
which is also $\Sigma^0_1$. 
It is clear that the case distinction is uniformly 
computable, so it follows that $f$ is computable.

\bigskip
Next, we check that $f$ is differentiable. We first note that every $f_n$ is differentiable, 
because $g$ is differentiable. 
Let $x \in [0,1]$. We distinguish two cases. 
First, if $x$ is in some $(q_n,r_n)$ then it is immediate that 
$f$ is differentiable at $x$ with derivative $f_n'(x)$, 
because the intervals $(q_n,r_n)$ are disjoint. 
Next, we consider the case where $x$ is not in any interval $(q_n,r_n)$. 
Note that in this case we have $f(x)=0$. Fix $m \in \omega$. Then we have:
\begin{align}
\notag\lim_{y \to x} \left|\frac{f(y)}{y-x}\right| &\leq 
\lim_{y \to x} \left|\frac{(f_0 + \dots + f_m)(y)}{y-x}\right| + 
\lim_{y \to x} \left|\frac{(f_{m+1} + f_{m+2} + \dots)(y)}{y-x}\right|.\\
\intertext{Because $f_0 + \dots + f_{m}$ is differentiable at $x$ with derivative $0$, this is equal to:}
&
\lim_{y \to x} \left|\frac{(f_{m+1} + f_{m+2} + \dots)(y)}{y-x}\right|.\label{eqn1}
\end{align}

To show that this limit is 0, we will prove that it is bounded by
$\frac{1}{2^{m}(1-x_0)}$ for every~$m$. 
Let $y \in [0,1]$ be distinct from $x$. Let us assume that $x < y$;  
the other case is proven in the same way. 
If $y$ is not in any $(q_n,r_n)$ for $n \geq m+1$ then $(f_{m+1} + f_{m+2} + \dots)(y)=0$. 
Otherwise, there is exactly one such $n$. Then:
\[\left|\frac{(f_{m+1} + f_{m+2} + \dots)(y)}{y-x}\right| = 
\left|\frac{f_n(y)}{y-x}\right| \leq \left|\frac{f_n(y)}{y-q_n}\right|,\]
where the last inequality follows from the fact that $x$ does not lie in $(q_n,r_n)$.
We distinguish three cases.
First, if $z = \frac{y-q_n}{r_n-q_n} \in (0,x_0]$, then
\[\left|\frac{f_n(y)}{y-q_n}\right| = \left|\frac{2^{-n}(r_n-q_n) g(z)}{y-q_n}\right|
= \left|2^{-n}z \sin(z^{-2})\right| 
\leq 2^{-n} \leq \frac{1}{2^{m}(1-x_0)}.\] 
Next, if $z \in [x_0,1-x_0]$ (which is nonempty because $x_0 \leq \frac{1}{2}$), then 
\[\left|\frac{f_n(y)}{y-q_n}\right| \leq \frac{2^{-n}(r_n-q_n)x_0^2}{y-q_n}
= \frac{2^{-n}x_0^2}{z} \leq  x_0 2^{-n} \leq \frac{1}{2^{m}(1-x_0)}\]
where we use the fact that $z \geq x_0$.
Finally, if $z \in [1-x_0,1]$, then
\[\left|\frac{f_n(y)}{y-q_n}\right| = \left|\frac{2^{-n}(r_n-q_n) h(1-z)}{y-q_n}\right|
\leq \frac{1}{2^n z} \leq \frac{1}{2^{n}(1-x_0)} \leq \frac{1}{2^{m}(1-x_0)}.\]
Combining this with \eqref{eqn1} we see that $\lim_{y \to x} \left|\frac{f(y)}{y-x}\right| \leq \frac{1}{2^{m}(1-x_0)}$. 
Since $m$ was arbitrary this shows that $f$ is differentiable at $x$, with derivative $f'(x)=0$.

\bigskip
Finally, we need to verify that $f'$ is discontinuous at $x$ for all $x \in V \setminus \mathrm{Int}(V)$. 
Therefore, let $x \in V \setminus \mathrm{Int}(V)$. 
Then every open set $W$ containing $x$ has nonempty intersection $W\cap U$ (recall that $U = [0,1] \setminus V$), 
but this intersection does not contain $x$. We have shown above that $f'(x) = 0$. 
We will show that for every open interval $I$ containing $x$ there is a point 
$y \in I$ such that $f'(y) \leq -1$, which clearly shows that 
$f'$ cannot be continuous at $x$. Fix an open interval $I$ containing $x$. 
Then $I\cap U\neq\emptyset$, so 
there is an $n \in \omega$ such that $I \cap [q_n,r_n]$ is nonempty. 
Note that $I$ contains $x$ and therefore $I$ cannot be a subinterval of $[q_n,r_n]$. 
Therefore there exists a $q_n < s < r_i$ such that either $[q_n,s) \subseteq I$ or $(s,r_n] \subseteq I$. We will assume the first case; the second case is proven in a similar way.

Note that on $[q_n,s)$ the function $f'$ is equal to $f_n'$. For $y \in (q_n,s)$ we thus have:
\[f'(y) = 2^{-n} g'((y-q_n)/(r_n-q_n)).\]
So, we need to show that there is a $y \in (q_n,s)$ such that $g'((y-q_n)/(r_n-q_n)) \leq -2^n$, or equivalently, that there is a $z \in (0,(s-q_n)/(r_n-q_n))$ such that $g'(z) \leq -2^n$. Without loss of generality, $(s-q_n)/(r_n-q_n) < x_0$. Let $k \geq n$ be such that 
$2^{-k} \leq \frac{s-q_n}{r_n-q_n}$. Then:
\begin{align*}
g'\left(1/\left(2^{k}\sqrt{\pi}\right)\right) &= \frac{1}{2^{k-1} \sqrt{\pi}} \sin(2^{2k}\pi) - 2^{k+1}\sqrt{\pi} \cos(2^{2k}\pi)\\
&=- 2^{k+1}\sqrt{\pi}
\leq - 2^k \leq -2^{n}.
\end{align*}
This completes the verification.
\end{proof}

\begin{thm}\label{thm-main2}
If $x \in [0,1]$ is such that every differentiable computable function 
$f: [0,1] \to \Real$ has continuous derivative at $x$, then $x$ is 1-generic.
\end{thm}
\begin{proof}
If $x$ is not 1-generic, then there is a $\Pi^0_1$ class $V$ such that 
$x\in V\setminus\mathrm{Int}(V)$. 
Applying Theorem~\ref{func-cons} to $V$ gives a differentiable computable function $f$ 
for which $f'$ is discontinuous at~$x$.
\end{proof}

\section{$n$-Genericity}

The notion of 1-genericity (Definition~\ref{def:1-generic}) 
corresponds to the first level of the arithmetical hierarchy. 
Higher genericity notions can be defined using forcing conditions 
from higher levels of the arithmetical hierarchy. 
As for 1-genericity, an equivalent formulation can be given as 
follows, see Jockusch~\cite{Jockusch}:
%Definition~\ref{def:1-generic} generalizes to:

\begin{defi} 
An element $x \in 2^\omega$ is \emph{$n$-generic} if $x$ meets every 
$\Sigma^0_n$ set of strings $A \subseteq 2^{<\omega}$ that is dense along $x$. 
\end{defi}

As usual, let $\emptyset'$ denote the halting set, and let 
$\emptyset^{(n)}$ denote the $n$-th jump. 
Since a $\Sigma^0_n$ set of strings is the same as a 
$\Sigma^0_1$ set of strings relative to $\emptyset^{(n-1)}$, 
a set is $n$-generic if and only if it is 1-generic relative to $\emptyset^{(n-1)}$. 

Corollary~\ref{cor-gen-int} relativizes to: 

\begin{prop}
For any $x \in 2^\omega$ we have that $x$ is n-generic if and only if
for every $\Pi^{0,\emptyset^{(n-1)}}_1$ class $V \subseteq 2^\omega$ we have 
$x \not\in V \setminus \mathrm{Int}(V)$.
\end{prop}

Note that in general a $\Pi^{0,\emptyset^{(n-1)}}_1$ class in $2^\omega$ is not 
the same as a $\Pi^0_n$ class, since the latter need not even be closed. 
(And even if one {\em assumes\/} that the class is closed the notions are not the same, 
see \cite[p76]{downey-hirschfeldt-2010}.) 

Given this equivalence, we can now generalize Definition~\ref{defi-gen-unit} to: 

\begin{defi} \label{def:n-generic}
Let $x \in [0,1]$. We say that $x$ is \emph{$n$-generic} if for every 
$\Pi^{0,\emptyset^{(n-1)}}_1$ class $V \subseteq [0,1]$ we have 
$x \not\in V \setminus \mathrm{Int}(V)$.
\end{defi}

Further justification for this definition comes from the fact that 
Proposition~\ref{prop-gen-int} relativizes: An irrational $x\in [0,1]$ 
is $n$-generic according to Definition~\ref{def:n-generic} 
if and only if its binary expansion is $n$-generic in $2^\omega$. 

It is straightforward to check that the results of all the previous 
sections relativize to an arbitrary oracle $A$. This gives the 
following relativized version of Theorem~\ref{thm-main}:

\begin{thm}
A real $x \in [0,1]$ is $1$-generic relative to $A$ if and only if for every
differentiable $A$-computable function $f: [0,1] \to \Real$, 
$f'$ is continuous at $x$.
\end{thm}

Taking $A=\emptyset^{(n-1)}$, this immediately gives the following 
characterization of $n$-genericity:

\begin{cor}
A real $x \in [0,1]$ is $n$-generic if and only if for every
differentiable $\emptyset^{(n-1)}$-computable function $f: [0,1] \to \Real$, 
$f'$ is continuous at $x$.
\end{cor}

Also, taking all $n$ together, we see that a real $x$ is 
arithmetically generic if and only if every differentiable 
arithmetical function has continuous derivative at~$x$.

\section{Multiply differentiable functions}\label{sec-mult-dif}

We have characterized $1$-genericity using the continuity of
the derivatives of (once) differentiable computable functions. One might wonder: what kind of effective
genericity for $x$ corresponds to every twice differentiable, computable function having continuous second derivative at $x$?
Or, more generally, what corresponds to every $n$ times differentiable, computable function having continuous $n$th derivative at $x$? It turns out
that the answer is always 1-genericity. To show this we will need the following proposition, which
essentially tells us that the case for $n > 2$ collapses to the case $n=2$.

\begin{prop}\label{c2-dif-comp}
Let $f: [0,1] \to \Real$ be computable and twice continuously differentiable. Then $f'$ is computable.
\end{prop}
\begin{proof}
See e.g.\ Pour-El and Richards \cite[Theorem 1.2]{Pour-ElRichards}.
\end{proof}

If the second derivative of a computable function exists, it is easy to see that it is of effective Baire class 2 (i.e.\ a pointwise limit
of a computable sequence of functions of effective Baire class 1), by similar arguments as in the proof of
Corollary \ref{dif-baire} However, using the following proposition we can easily see that the second derivative of a computable function
is in fact of effective Baire class 1.

\begin{prop}\label{dif2-baire}
Let $f: [0,1] \to \Real$ be twice differentiable. Then
\[f''(x) = \lim_{h \to 0} \frac{f(x+h) + f(x-h) - 2f(x)}{h^2}.\]
\end{prop}
\begin{proof}
See e.g.\ Rudin \cite[p.\ 115]{Rudin}.
\end{proof}

\begin{thm}
Fix $n \geq 1$. Then a real $x \in [0,1]$ is $1$-generic if and only if every $n$ times differentiable, computable function $f: [0,1] \to \Real$ has continuous $n$th derivative at $x$.
\end{thm}
\begin{proof}
For $n=1$ this is exactly Theorem \ref{thm-main}. So, we may assume $n \geq 2$. First, if $x \in [0,1]$ is not $1$-generic, then by Theorem \ref{thm-main} there is a differentiable, computable function $g: [0,1] \to \Real$ such that $g'$ is not continuous at $x$. Now let $h_1 = g$ and let $h_i$ be a computable antiderivative of $h_{i-1}$ for $2 \leq i \leq n$ 
%(which exist by Pour-El and Richards \cite[Theorem 3.1]{Pour-ElRichards}). 
(which exists by Ko~\cite[Theorem 5.29]{Ko}). 
Then, if we let $f = h_n$, we see that $f$ is an $n$ times differentiable, 
computable function such that $f^{(n)}$ is discontinuous at $x$.

Conversely, if $f$ is an $n$ times differentiable, computable function, then $f^{(n-2)}$ is computable by Proposition \ref{c2-dif-comp}. So, $f^{(n)}$ is of effective Baire class $1$ by Proposition \ref{dif2-baire}. Thus, if $f^{(n)}$ is discontinuous at $x$, then $x$ is not 1-generic by Theorem \ref{baire-eff}.
\end{proof}

\section{Complexity theoretic considerations}\label{sec-poly}

In this section we discuss polynomial time computable real functions. 
The theory of these functions is developed in Ko~\cite{Ko}, to which 
we refer the reader for the basic results and definitions. 
Briefly, a function $f: [0,1] \to \Real$ is polynomial time computable
if for any $x\in [0,1]$ we can compute an approximate value of $f(x)$
to within an error of $2^{-n}$ in time $n^k$ for some constant $k$. 

Most of the common functions from analysis, such as rational functions 
and the trigonometric functions, as well as their inverses, are all 
polynomial time computable, 
see e.g.\ Brent~\cite{Brent} and Weihrauch~\cite{Weihrauch}. 
Also, the polynomial time computable functions are closed under 
composition. 
With this knowledge, it is not difficult to see that the construction 
of the function $f$ in section~\ref{sec:func-cons} can be modified to yield 
a polynomial time computable function, rather than just a computable one. 
For this it is also needed that the complement of the 
$\Pi^0_1$ class $V$ from Theorem~\ref{func-cons} can be represented by a 
polynomial time computable set of strings. 
This is similar to the fact that every nonempty computably enumerable set 
is the range of a polynomial time computable function, simply by 
sufficiently slowing down the enumeration.
Since the enumeration of $U = [0,1] \setminus V$ in the proof of 
Theorem~\ref{func-cons} is now slower, the definition of $f_n$ in 
\eqref{eq:f} has to be adapted by replacing $2^n$ by $2^{t(n)}$, 
where $t(n)$ is the stage at which the interval $(q_n,r_n)$ is 
enumerated into~$U$. This modification ensures that the functions 
$f_n$ are uniformly polynomial time computable, so that also the function 
$f = \sum_{n=0}^\infty f_n$, is polynomial time computable. 
Thus we obtain the following strengthening of Theorem~\ref{func-cons}:

\begin{thm}\label{thm:modified}
Let $V$ be a $\Pi^0_1$ class.
Then there exists a differentiable polynomial time computable function
$f: [0,1] \to \Real$ such that $f'$ is discontinuous at every
$x \in V \setminus \mathrm{Int}(V)$.
\end{thm}

We now have the following variant of Theorem~\ref{thm-main}:

\begin{thm}
A real $x \in [0,1]$ is 1-generic if and only if for every
differentiable polynomial time computable function $f: [0,1] \to \Real$,
$f'$ is continuous at $x$.
\end{thm}
\begin{proof}
The ``only if'' direction is immediate from Theorem~\ref{thm-main}. 
For the ``if'' direction;
if $x$ is not 1-generic, then there is a $\Pi^0_1$ class $V$ such that 
$x\in V\setminus\mathrm{Int}(V)$. 
Theorem~\ref{thm:modified} then gives a differentiable polynomial time 
computable function $f$ for which $f'$ is discontinuous at~$x$.
\end{proof}

\providecommand{\bysame}{\leavevmode\hbox to3em{\hrulefill}\thinspace}
\providecommand{\MR}{\relax\ifhmode\unskip\space\fi MR }
\providecommand{\MRhref}[2]{%
  \href{http://www.ams.org/mathscinet-getitem?mr=#1}{#2}
}
\providecommand{\href}[2]{#2}

\end{document}